\documentclass{article}

\usepackage{amsthm}
\usepackage{amsmath}
\usepackage{amssymb}
\usepackage{verbatim}

\newtheorem{thm}{Theorem}
\newtheorem{lem}{Lemma}

\newtheorem{clm}{Claim}

\title{A PL-manifold of nonnegative curvature homeomorphic to $S^2 \times S^2$ is a direct metric product
\\
{\em Preliminary version}}

\author{Sergey Orshanskiy}
\date{Spring 2008}

\begin{document}

\maketitle


\begin{abstract}

Let $M^4$ be a PL-manifold of nonnegative curvature that is homeomorphic to a product of two
spheres, $S^2 \times S^2$.
We prove that $M$ is a direct metric product of two spheres endowed with some
polyhedral metrics. In other words, $M$ is a direct metric product of the surfaces of two
convex polyhedra in $\mathbb{R}^3$.

The background for the question is the following.
The classical H.Hopf's hypothesis states: for any Riemannian metric on 
$S^2 \times S^2$ of nonnegative sectional curvature the curvature cannot be strictly positive at all points.
There is no quick answer to this question: it is known that a Riemannian metric on $S^2 \times S^2$ of
nonnegative sectional curvature need not be a product metric. However, 
M.Gromov has pointed out that the condition of nonnegative curvature in the PL-case appears to be stronger
than nonnegative sectional curvature of Riemannian manifolds
and analogous to some condition on the curvature operator. So the motivation for the question addressed in
this text is to settle the PL-version of the Hopf's hypothesis.

\end{abstract}

This paper presents a structure result for polyhedral $4$-manifolds with curvature bounded from below. 
The classical H.Hopf's hypothesis states: for any Riemannian metric on 
$S^2 \times S^2$ of nonnegative sectional curvature the curvature cannot be strictly positive at all points.
There is no quick answer to this question: it is known that a Riemannian metric on $S^2 \times S^2$ of
nonnegative sectional curvature need not be a product metric. The distinction between a condition
on the sectional curvature and a condition on the curvature operator is important.
While it is known that a $4$-manifold with a positive curvature operator is homeomorphic to a sphere,
positive sectional curvature for a $4$-manifold only implies that the fundamental
group is either $\mathbb{Z}^2$ or trivial \cite{bourgignon-75, kuranishi-90}.
For more references, \cite{wilking-2007} is an
extensive survey of similar results for positive and nonnegative curvature.

According to \cite{cheeger-86},
M.Gromov has pointed out that the condition of nonnegative curvature in the PL (polyhedral)-case appears to be stronger than nonnegative sectional curvature of Riemannian manifolds and analogous to some condition on the curvature operator.
As a confirmation of this remark, this paper solves the polyhedral case of the H.Hopf conjecture: a PL-manifold
of nonnegative curvature homeomorphic to $S^2 \times S^2$ is a direct metric product.

To fix the terminology: a PL-manifold is a locally finite simplicial complex all whose simplices are metrically flat
(convex hulls of finite sets of points in a Euclidean space) that is
also a topological manifold. In the compact case, ``locally finite'' implies
``finite'', so we are working with some finite simplicial decomposition.

A (metric) singularity in a PL-manifold $M^n$ is a point $x \in M$ that has no flat neighborhood.
Metric singularities comprise $M_s$, the singular locus. $M \backslash M_s$ is a flat Riemannian manifold.
More specifically, a singularity of codimension $k$ has a neighborhood that is a direct metric product
of an open set in $\mathbb{R}^{n - k}$ with another space, yet no such product for $\mathbb{R}^{n - k + 1}$.
We will be interested in the case when $M$ is also an Alexandrov space of nonnegative curvature.
This condition is known to be equivalent to the following formulation:
the link of $M_s$ at each singularity of codimension $2$ is a circle of length
$< 2\pi$.

Given $M$, a PL-manifold of nonnegative curvature homeomorphic to $S^2 \times S^2$. The claim is that $M$ is a direct metric product.
The proof is carried out in two stages.
Firstly, we establish the existence of two parallel distributions of oriented 2-planes $\alpha$ and
$\beta$ (2-distributions for brevity), foliating  $M \backslash M_s$ and orthogonal to each other.
Secondly, we use these fields of planes to decompose $M$ into a direct metric product and argue that the factors are convex polyhedra in $\mathbb{R}^3$.

\subsection*{Acknowledgements}
I am grateful to my thesis adviser, Professor Dmitri Burago for introducing me to the questions addressed
in this article and specifically for pointing out the importance of Cheeger's results.

\section{Finding parallel 2-distributions}

In this section we remove all singularities from our consideration and
focus on $M \backslash M_s$, a flat Riemannian manifold. The goal is to
find two parallel distributions of oriented 2-planes on $M \backslash M_s$.
This is the main step towards factoring $M$ as a direct metric product.
The main tool here are J.Cheeger's results for polyhedral spaces of nonnegative
curvature. His results are stated in the language of differential forms
(and this is why we are focusing on $M \backslash M_s$, as differential
forms on $M$ are not well-defined).

Since $M \backslash M_s$ is a flat Riemannian manifold, one can indeed study differential forms on it.
Every parallel form (i.e. $\nabla \omega = 0$) on $M \backslash M_s$ is harmonic, $L_2$, closed and
co-closed, as is verified by taking the differential and the codifferential in local (flat)
coordinates and integrating in local coordinates (there is a finite flat atlas coming from the
PL-structure).

The situation is considerably better because of J. Cheeger's results for (in particular) PL-manifolds
of nonnegative curvature \cite{cheeger-86}. We are using his results in the following form:
\begin{thm}[J.Cheeger]
Let $M^n$ be a PL-manifold of nonnegative curvature. Let $H^i$ be the space of
$L_2$-harmonic forms on $M \backslash M_s$ that are closed and co-closed. Then $dim H^i = b^i(M)$.
Moreover, all forms in $H^i$ are parallel.
\end{thm}


What it means for our present discussion, given that $b^2(S^2 \times S^2) = 2$, is that the vector space of parallel forms on $M \backslash M_s$ is $2$-dimensional. Pick a basis $\{\omega_1, \omega_2\}$ for this vector space.
Thus, we have $\omega_1$ and $\omega_2$ --- two parallel 2-forms on $M \backslash M_s$ that are linearly independent (not proportional to each other). We are going to do some linear algebra with these forms in order to
obtain two mutually orthogonal 2-distributions (parallel fields of oriented 2-planes). This will prove
\begin{clm}
\label{clm:two-fields-of-planes}
Let $M$ be a PL-manifold of nonnegative curvature, homeomorphic to 
$S^2 \times S^2$ and let $M_s$ be the singular locus of $M$.
Then there are two mutually orthogonal parallel fields of oriented 2-planes
on $M \backslash M_s$.
\end{clm}

\begin{proof}[Proof of claim]
A parallel 2-form on a flat Riemannian manifold has a well-defined notion of
eigenvalues. If we can find a parallel antisymmetric 2-form on $M \backslash M_s$ with four distinct eigenvalues, $\pm ai \neq \pm bi$, this form immediately
gives rise to two fields of planes. One of the (fields of) planes is given by $\{v \in T_x M~|~ \exists w:~\omega(v, w) = \max(a, b)\cdot\|v\|\cdot\|w\|\}$. The other is the orthogonal complement (and also the eigenspace corresponding to the smaller eigenvalue). Since $a$ and $b$ cannot both be $0$, one of the planes acquires orientation from the form $\omega$ itself. The other can be
oriented using the orientation of its orthogonal complement.

Let $\omega_1$ and $\omega_2$ be the two 2-forms on $M \backslash M_s$ not proportional to each other, that came from the Cheeger's results for nonnegative curvature. Assume for a contradiction that all real linear
combinations of these two forms have repeating eigenvalues (otherwise we would
immediately obtain two families of planes, as desired). The contradiction
will become clear from simple linear algebra done in local flat coordinates.

The four following lemmas (\ref{lem:antisymmetric-multiple-orthogonal},
\ref{lem:orthogonal-antisymmetric-two-kinds}, 
\ref{lem:linear-first-second-repeating-eigenvalues}, 
\ref{lem:commuting-two-first-kind-su2}) are technical and straightforward, and
are only used to prove Claim~\ref{clm:two-fields-of-planes}.

\begin{lem}
\label{lem:antisymmetric-multiple-orthogonal}
If $\omega$ is an antisymmetric 2-form with repeating eigenvalues defined on a $4_{\mathbb{R}}$-dimensional
vector space, then  its matrix is a scalar multiple of an orthogonal matrix.
\end{lem}
\begin{proof}
If the form is nonzero, rescale it to make the eigenvalues equal to $\pm i$
(each with multiplicity $2$). The resulting form is given by a matrix $A$.
The form is antisymmetric, so $A^T = -A$. Therefore $A$ can be diagonalized via
some unitary matrix, $U^\ast A U = 
\left[
\begin{array}
{cccc}
i & 0 & 0 & 0 \\
0 & i & 0 & 0 \\
0 & 0 & -i & 0 \\
0 & 0 & 0 & -i
\end{array}
\right] \in U(4)$

So $A$ is unitary itself and real-valued, hence orthogonal.
\end{proof}

\begin{lem}
\label{lem:orthogonal-antisymmetric-two-kinds}
If $A$ is a $4 \times 4$ real-valued matrix that is also orthogonal and antisymmetric, then either

$A = 
\left[
\begin{array}{rrrr}
0 & a & b & c \\
-a & 0 & c & -b \\
-b & -c & 0 & a \\
-c & b & -a & 0
\end{array}
\right]$ (a matrix of the first kind)

or

$A = 
\left[
\begin{array}{rrrr}
0 & a & b & c \\
-a & 0 & -c & b \\
-b & c & 0 & -a \\
-c & -b & a & 0
\end{array}
\right]$ (a matrix of the second kind)

for some real numbers $a, b, c$ satisfying $a^2 + b^2 + c^2 = 1$.
\end{lem}
\begin{proof}
The proof is straightforward. Start from an orthogonal matrix of the form
$\left[
\begin{array}{rrrr}
0  &  a & b  & c \\
-a &  0 & d  & e \\
-b & -d & 0  & f \\
-c & -e & -f & 0
\end{array}
\right]$

The columns are normalized:
\begin{eqnarray*}
  a^2 + b^2 + c^2 = 1 \\
  a^2 + d^2 + e^2 = 1 \\
  b^2 + d^2 + f^2 = 1 \\
  c^2 + e^2 + f^2 = 1
\end{eqnarray*}

Consequently, $a^2 = f^2$, $b^2 = e^2$, and $c^2 = d^2$.

The columns are orthogonal to each other:
\begin{eqnarray*}
  bd + ce = 0 \\
  ad = cf \\
  ae + bf = 0 \\
  ab + ef = 0 \\
  ac = df \\
  bc + de = 0
\end{eqnarray*}

Recall that $a^2 = f^2$.
The case $a = f = 0$ is easy. Assume this is not the case.
If $a = f$, then $c = d$ and $b = -e$, so the matrix is of the first kind.
If $a = -f$, then $b = e$ and $ c = -d$, and the matrix is of the second kind.
The two kinds are easily seen to be mutually exclusive.
\end{proof}

\begin{lem}
\label{lem:linear-first-second-repeating-eigenvalues}
Let $J =
\left[
\begin{array}{cccc}
0 & 1 & 0 & 0 \\
-1 & 0 & 0 & 0 \\
0 & 0 & 0 & 1 \\
0 & 0 & -1 & 0
\end{array}
\right]$, a particular matrix of the first kind.
If $B$ is a matrix of the second kind (as above) and $\lambda, \mu$ are both nonzero real numbers, then $\lambda J + \mu B$ is an antisymmetric matrix with distinct eigenvalues. 
\end{lem}
\begin{proof}
$\lambda J + \mu B =
\left[
\begin{array}{rrrr}
0 & \lambda + \mu a & \mu b & \mu c \\
-\lambda - \mu a & 0 & -\mu c & \mu b \\
-\mu b & \mu c & 0 & \lambda - \mu a \\
-\mu c & -\mu b & -\lambda + \mu a & 0
\end{array}
\right]$

This matrix is antisymmetric and it is never a multiple of an orthogonal matrix
(given that $\lambda, \mu \neq 0$). If $a \neq 0$, compare the norms of different columns. If $a = 0$ and $b \neq 0$, take the dot product of the first column with the fourth column. If $a = 0$ and $c \neq 0$, take the dot product
of the first column with the third column. 
Either way, the matrix has to have distinct eigenvalues by Lemma~\ref{lem:antisymmetric-multiple-orthogonal}
\end{proof}

\begin{lem}
\label{lem:commuting-two-first-kind-su2}
Let $J$ be as above and let $C$ be another matrix of the first kind, not
a multiple of $J$ ($a \neq \pm 1)$. Let $G \subset O(4)$ be the group
of all orthogonal matrices commuting with $J$ and $C$:
$G = \{A | A \in O(4), AJ = JA, AC = CA\}$. Then $G = SU(2)\subset SO(4)$. 
\end{lem}
\begin{proof}
For notational convenience, certain $2 \times 2$ matrices can be abbreviated as complex numbers:
$\left[
\begin{array}{rr}
a & -b \\
b & a
\end{array}
\right]
\leftrightarrow
a + bi.
$
Also let $\beta =
\left[
\begin{array}{rr}
0 & 1 \\
1 & 0
\end{array}
\right]$.

In this notation, a matrix of the first kind
$C = 
\left[
\begin{array}{rrrr}
0 & a & b & c \\
-a & 0 & c & -b \\
-b & -c & 0 & a \\
-c & b & -a & 0
\end{array}
\right]
=
\left[
\begin{array}{cc}
-ia & (c - ib)\beta = \beta(c + ib) \\
(-c + ib) \beta = \beta (-c - ib) & -ia 
\end{array}
\right]$.

$AJ = JA$ is equivalent to saying that
$A = 
\left[
\begin{array}{rr}
z_1 & z_2 \\
z_3 & z_4 
\end{array}
\right]$ for some four complex numbers
$z_1, z_2, z_3, z_4$. We are also given that $AC = CA$.

$AC = 
\left[
\begin{array}{rr}
-z_1 ia + z_2(-c + ib) \beta & z_1 (c - ib) \beta - z_2 ia \\
-z_3 ia + z_4(-c + ib) \beta & z_3 (c - ib) \beta - z_4 ia
\end{array}
\right] =
$

$CA = \left[
\begin{array}{rr}
-z_1 ia + \beta (c + ib) z_3 & -z_2 ia + \beta (c + ib) z_4 \\
\beta (-c - ib) z_1 - z_3 ia & \beta (-c - ib) z_2 - z_4 ia
\end{array}
\right]$

If $z$ is a complex number, then clearly $\beta z \beta = \overline{z}$ --- the complex conjugate of $z$. Also notice that $\beta^2 = 1$.
Then $AC = CA$ is equivalent to four conditions:

\begin{eqnarray*}
\overline{z_2 (-c + ib)} = (c + ib) z_3 \\
\overline{z_1 (c - ib)} = (c + ib) z_4 \\
\overline{z_4 (-c + ib)} = (-c - ib) z_1 \\
\overline{z_3 (c - ib)} = (-c - ib) z_2
\end{eqnarray*}

Equivalently, $\overline{z_2} = -z_3$ and $\overline{z_1} = z_4$
(we assumed that $a \neq \pm 1$, so $-c \pm ib \neq 0$). The orthogonality of $A$ gives the normalization: $|z_1|^2 + |z_2|^2 = |z_3|^2 +
|z_4|^2 = 1$. Lastly, $G = SU(2)\subset SO(4)$ is precisely the set of matrices
of the form
$\left[
\begin{array}{rr}
z_1 & z_2 \\
-\overline{z_2} & \overline{z_1}
\end{array}
\right]$ satisfying $|z_1|^2 + |z_2|^2 = 1$.

\end{proof}

Any parallel form is preserved by holonomies. Writing down $\omega_1$ and $\omega_2$ in local flat coordinates (and rescaling both of them, if necessary) one obtains two orthogonal antisymmetric $4 \times 4$ matrices
$\Omega_1$ and $\Omega_2$ that are not proportional to each other (by 
Lemma~\ref{lem:antisymmetric-multiple-orthogonal} --- recall that by assumption both forms have repeating eigenvalues and so do not immediately give us a parallel field of oriented 2-planes). By classification from 
Lemma~\ref{lem:orthogonal-antisymmetric-two-kinds}, there are two kinds of such matrices. Without loss of generality, $\Omega_1 = J$ 
--- a particular matrix of the first kind. 
Indeed, one can make an appropriate orthogonal change of coordinates, as $\Omega_1$ is a real-valued matrix of a normal operator with imaginary eigenvalues.
If $\Omega_2$ is a matrix of the second kind,
then $\lambda \Omega_1 + \mu \Omega_2$ has distinct eigenvalues for $\lambda, \mu \neq 0$ (Lemma~\ref{lem:linear-first-second-repeating-eigenvalues}). 
Otherwise note that all scalar multiples of the matrices of the first kind constitute a $3$-dimensional subspace of all real-valued matrices $4 \times 4$.
Choose $\Omega_3$ that is linearly independent with the previous two matrices
and appropriately rescale it in order to make it orthogonal and so a matrix of the first kind, too.

Parallel forms $\omega_1$ and $\omega_2$ are preserved
by holonomies and so any matrix in the image of the holonomies in $SO(4)$ commutes with 
$\Omega_1$ and $\Omega_2$.
Any orthogonal matrix commuting with $\Omega_1$ and $\Omega_2$ is in
$SU(2)$ and so has to commute with $\Omega_3$ (use Lemma~\ref{lem:commuting-two-first-kind-su2} two times).
Then any matrix in the image of the holonomies in $SO(4)$ has to commute with $\Omega_3$
as well.
We can obtain a third parallel antisymmetric 2-form $\omega_3$ 
from the form given by the matrix $\Omega_3$ by parallel-translating it to all other points
of $M \backslash M_s$.
The new form $\omega_3$ is
linearly independent with the previous two, leading to the desired
contradiction. Indeed, it has already been established that the space of such forms is $2$-dimensional as a consequence of $b^2(S^2 \times S^2) = 2$.
This proves Claim~\ref{clm:two-fields-of-planes} about the existence of
the desired 2-distributions (two parallel fields of oriented 2-planes
on $M \backslash M_s$ orthogonal to each other).
\end{proof}

These two 2-distributions allow us to give a more specific description of
all parallel 2-forms on $M \backslash M_s$.
Clearly, the signed areas of the projections onto the first
and the second of the planes that we have just found constitute two parallel degenerate 2-forms
that are not proportional to each other. Hence, they span $H^2(M)$. 
In appropriate local coordinates these two forms are just $dx_1 \wedge dx_2$ and
$dx_3 \wedge dx_4$, respectively. Two 2-distributions can be thought of
as the kernels of these two forms. The sum of these two forms,
$dx_1 \wedge dx_2 + dx_3 \wedge dx_4$ is a symplectic form with 
repeating eigenvalues $\pm i$. This form yields a pseudocomplex structure on
$M \backslash M_s$, so

\begin{lem}
\label{lem:M-is-PL-Kahler}
$M$ is a polyhedral K\"ahler manifold (see \cite{panov-06} for the definition).
\end{lem}
\begin{proof}
The matrices representing the holonomies preserving the form $dx_1 \wedge dx_2 +
dx_3 \wedge dx_4$ commute with $J$ (in the appropriate positively oriented orthogonal basis, where $J$ is the matrix of the 2-form $dx_1 \wedge dx_2 +
dx_3 \wedge dx_4$).
Commuting with $J$ is equivalent to being in $GL(2, \mathbb{C})$. However,
$O(4) \cap GL(2, \mathbb{C}) = U(2)$. (If the basis turns out to be negatively
oriented, use $dx_1 \wedge dx_2 - dx_3 \wedge dx_4$ instead.)
Thus, the image of the holonomies of $M$
is in a subgroup of $SO(4)$, conjugate to $U(2)$ --- precisely what the definiton of a polyhedral K\"ahler manifold says. 
(Note that this image is in a subgroup conjugate to $U(2)$, \textbf{but not} to $SU(2)$!)
\end{proof}


\section{Decomposing $M$ into a product}

The two 2-distributions we have found behave nicely, but are defined only on the
nonsingular part $M \backslash M_s$. While we expect $M \backslash M_s$ to be a direct product too, it is easier to factor $M$ as a whole. 
Fortunately, these two 2-distributions allow us to discern the product structure at the singularities from
$M_s$ and, moreover, the 2-distributions turn out to be parallel (respectively, perpendicular) to
the fibers of this local product structure.

\subsection{Classifying singularities}
If you want to skip the details of this local analysis, go straight
to the conclusion (Lemma~\ref{lem:local-product-decomposition-of-M}).

Since $M$ is a polyhedral K\"ahler manifold as established above
(Lemma~\ref{lem:M-is-PL-Kahler}), we can use the following result from \cite{panov-06}:
\begin{lem}[D.Panov]
\label{lem:no-codimension-3}
Let $M^4$ be a $4_{\mathbb{R}}$-dimensional polyhedral K\"ahler manifold. Then there are no singularities of codimension $3$ (all singularities have to have codimension $2$ or $4$).
\end{lem}
\begin{proof}
The proof uses the fact that the singular locus of $M$ is a holomorphic subspace of $M$ in the sense of K\"ahler structure, and also some Morse theory.
See \cite[proposition 3.3]{panov-06}.
\end{proof}

There can only be finitely many singularities of codimension $4$
(they all have to be vertices of the simplicial decomposition of $M$).
The locus of singularities of codimension $2$ with the induced intrinsic metric is a flat $2$-dimensional manifold (since near every singularity of codimension $2$ $M$ can be decomposed into a product of a flat space with a $2$-cone) that is also a subset of the two-skeleton $\Sigma^2$. It remains to add that codimension $4$ singularities cannot be isolated from the rest of the singular locus.

\begin{lem}
\label{lem:flat-pinched-no-singularity}
A singular point $x \in M_s$ of any codimension cannot be isolated (from the rest of the singular locus).
In other words, if a point in a $4$-dimensional PL-manifold has a flat pinched neighborhood, this point is not a singularity (has a flat neighborhood).
\end{lem}
\begin{proof}
The link at $x$ is homeomorphic to $S^3$, thus simply connected. It is also
a space of curvature $\kappa \ge 1$ because $M$ itself has nonnegative curvature. Moreover, the link at $x$ is a space of curvature $\kappa \le 1 + \epsilon$ for any $\epsilon > 0$. This is true, since a sufficiently small triangle in the link witnessing $\kappa \nleq 1 + \epsilon$ would also witness that a pinched neighborhood of $x$ (that is supposed to be flat) is not a space
of nonpositive curvature --- a contradiction. 

Because of the properties of Alexandrov spaces of curvature bounded from above,
$1 \le \kappa \le 1 +\epsilon$ for any $\epsilon > 0$ does imply that
$\kappa \equiv 1$ for the link. So the link at $x$ is a simply connected
space of constant curvature $1$, thus the standard $3$-sphere. The cone
over this sphere is $\mathbb{R}^4$ and so $x$ is not a singularity, as claimed.
\end{proof}

Note that this argument fails in the two-dimensional case ($S^1$ is not simply connected). Indeed, singularities of a $2$-dimensional PL-manifold are always of the highest possible codimension and yet always isolated from one another.

At this point, we can conclude that the singular locus of $M$ consists of several triangles that are also faces in the simplicial decomposition of $M$,
where the vertices of these triangles may be singularities of codimension $4$,
but at all other points the singular locus is a flat $2$-dimensional manifold.
It turns out that the singularities have to be aligned in accord with the two parallel 2-distributions that we have just found.

\begin{lem}
\label{lem:codim-2-singularity-aligned}
Let $x \in M_s$ be a singularity of codimension $2$. By definition, $M$ can be factored near this singularity as $C \times \mathbb{R}^2$ (this factoring is unique --- just take all geodesics passing through the origin).
Then the fibers of this factoring are parallel (respectively, perpendicular) to the two 2-distributions found above. 
\end{lem}
\begin{proof}
Consider an Euler vector field near $x$ stretching the metric away from the
singular locus. This vector field is parallel to the conical fibers and is directed away from the vertex of such a conical fiber. Let $\omega$ be 
a $2$-form such that its kernel is one of our 2-distributions.
$\omega$ is a parallel form, so in particular it is preserved by the holonomy, resulting from going around $x$ any number of times. Take any nonsingular point near $x$ and two tangent vectors,
parallel to one and (respectively) the other of the fibers of the unique local
product structure near $x$, say $u$ and $v$.

After going around $x$, $u$ is unchanged, but $v$ is turned by some angle and becomes $\hat{v}$. Yet $\omega(u, v) = \omega(u, \hat{v})$, so
$\omega(u, v - \hat{v}) = 0$. By choosing $v$ appropriately we can make
$v - \hat{v}$ to have any direction in its plane. Thus, 
$\omega(u, v) = 0$ if $u$ and $v$
are parallel to different fibers ($(C, \ast)$ and $(\ast, \mathbb{R}^2)$). Given 
that the kernel of $\omega$ is $2$-dimensional, it is easy to see that the kernel is indeed parallel to one of the fibers.
\end{proof}

\begin{lem}
\label{lem:classifying-codim-4-singularity}
Let $x \in M_s$ be a singularity of codimension $4$. Then this singularity
can be factored as a product of two conical singularities, aligned in accord with our parallel 2-distributions of oriented planes.
\end{lem}
\begin{proof}
Recall that a codimension $4$ singularity cannot be isolated (Lemma~\ref{lem:flat-pinched-no-singularity}). As we know from the work of Dmitri Panov, in a $4$-dimensional polyhedral K\"ahler manifold we cannot have any singularities of codimension $3$ (Lemma~\ref{lem:no-codimension-3}). The locus of singularities of codimension $2$
is a flat $2$-dimensional manifold (as follows from the unique factoring of a codimension $2$ singularity), and also this locus is a subset of the $2$-skeleton of $M$. All that implies that the singular locus near $x$ is $x$ itself and several (finitely many) connected singular components of codimension $2$, all looking like pinched cones with $x$ in the center. 
 
Assume there is only one such connected component. Recall that it has to be parallel to one of our 2-distributions (Lemma~\ref{lem:codim-2-singularity-aligned}). All holonomies near $x$ can be generated by going around singularities of codimension $2$. In particular, the tangent vectors from $T(M \backslash M_s)$ that are parallel to the other 2-distribution are preserved by all holonomies
near $x$. So we can find a parallel vector field
near $x$ that is also parallel to one of the $2$-distributions and hence to the
singular locus. Integrating this field we obtain infinite geodesics, so we can
split the singularity by the Splitting Theorem for Alexandrov spaces of nonnegative curvature, contradicting the assumption that $x$ has codimension $4$. This proves that there are at least two connected components of the singular locus of codimension $2$ and moreover that there is at least one component parallel to one $2$-distribution and at least one parallel to the other.

Now consider the link of $M$ at $x$ viewed as all rays emanating from $x$ in a cone over the link 
itself. 
The singular rays (by definition) are those without a neighborhood in the link isometric to
a piece of the standard $3$-sphere. Clearly, the singular rays comprise several
circles. These circles are closed geodesics. This can be checked locally at any point as all points
of a circle look the same, and so it suffices to notice that the locus of singularities in $M$ of codimension $2$ 
consists of singular triangles (except, perhaps, their vertices) that are simplicial faces of $M$ and thus 
totally geodesic. Moreover, the link is an Alexandrov space of curvature $\kappa \ge 1$.
Therefore, the distance from every ray (point in the link) to every singular circle is at most $\pi / 2$. If the distance between a ray $r$ and a singular circle is exactly $\pi / 2$, then this is the distance from $r$ to every ray compsiring this singular circle. 

The distance between two singular circles parallel to different 2-distributions
is precisely $\pi / 2$. This is because every shortest path in a neighborhood of $x$ in $M$ from a point on one
ray emanating from $x$ (belonging to one of the singular circles in the link) to a point on another ray emanating from $x$ (belonging to the other singular circle in the link) has
to pass through $x$. Otherwise this path would have to be orthogonal to one of the 2-distributions and parallel to the other, so anyway start from going along the original ray directly to $x$.  

The distance between any two rays is strictly less than $\pi$ (no geodesic passes through $x$ as the singularity is of codimension $4$ --- again using the Splitting Theorem for Alexandrov spaces of nonnegative curvature). So shortest paths in $M$ near $x$ correspond
to shortest paths between the corresponding rays in the link in the obvious
way (as in any cone of diameter $< \pi$). Notice also that singular points (from
$M_s$) near $x$ in $M$ belong to singular rays (in the sense of having no neighborhood in the link isometric to a piece of a standard sphere), while nonsingular points belong to nonsingular rays. 
$M$ is a PL-manifold of nonnegative curvature, so any shortest path in $M$ between two nonsingular points
consists entirely of nonsingular points. Consequently, any shortest path in the link between two
nonsingular rays may only contain nonsingular rays. Similarly, any shortest path between two
singular rays has only singular or only nonsingular intermediate rays. 

Some singular and some nonsingular rays are parallel to one of the 2-distributions. If two rays $r_1$ and $r_2$ are parallel to the same 2-distribution, so are all rays in any shortest path connecting them. Indeed, just consider the angular 
sector consisting of all rays in this path --- the angle between the rays is less than $\pi$ so the plane of the sector is precisely the plane of the
2-distribution in question. To put it more abstractly, all rays parallel to either of the 2-distributions form a convex subspace
of the link. One corollary is that the singular circles
are totally geodesic subspaces of the link. Another one is that there are only
two singular circles in the link (otherwise take two so that their rays are parallel to the same 2-distribution and connect any two rays from them by a path in the link).
Similary, no nonsingular ray can be parallel to either 2-distribution.
So the distance in the link from any nonsingular ray $r$ to either of the
two singular circles (any to any ray in these circles) is strictly less than
$\pi / 2$. Indeed, if it were $\pi / 2$ (to the closest and hence all rays in a singular circle --- recall
the remark about closed geodesics), the 
ray $r$ would be orthogonal to one and so parallel to the other 2-distribution.

All this implies that the singular locus near $x$ consists of $x$ itself and
two components of codimension $2$ that are pinched cones over the two
singular circles in the link. Taking any point near $x$, we can project it onto
both components (by finding the closest point). Since this is the same as finding the closest ray in a singular circle to a given ray, we can conclude that having $x$ as one of the two projections is equivalent to being a singular
point. Moreover, both projections are unique (and so the operations of taking both projections are well-defined). Indeed, asume that there were a nonsingular point $p$ near $x$ with two shortest paths from $p$ to one of the singular components, say $[pu]$ and $[pu']$. The angle between these two paths at $p$
cannot be $\pi$: we are looking for shortest paths from a point in a cone to some set of rays in this cone; any such shortest path should start from locally
decreasing the radial distance (coordinate along the rays of the cone) so two
shortest paths cannot run in the opposite directions. But if the angle is less than $\pi$, we use the same argument as before: $[pu]$ and $[pu']$ are orthogonal
to the same 2-distribution, hence $upu'$ defines a plane (leaf) that is parallel
to the other 2-distribution yet does not pass through $x$ ($p \in M$ is nonsingular, so the ray $[xp)$ is not parallel to either 2-distribution). It allows us to move $p$ along the
bisector of $upu'$, decreasing the distance from $p$ to the singular component in question and yet keeping the nonuniqueness of a shortest path. 
Eventually $p$ runs into the singular locus, but $upu'$ does not pass through $x$, so $p$ will run into a codimension $2$ singularity orthogonal to $upu'$ (from the component to which we are measuring distances). Yet when $p$ is close
to this singular component, the uniqueness of a shortest path is clear ---
a contradiction.

Therefore, we get a well-defined continuous mapping from a conical neighborhood of $x$ in $M$ into a product of two cones (via two projections). It sends $x$ to
the origin in the product and the rest of the singular locus into the two
cones (factors) in the product. Clearly, at any nonsingular point near $x$ in
$M$ this mapping is a local isometry (use 2-distributions). Now consider
a codimension $2$ singular point $p$ near $x$. What happens with this
mapping near $p$ --- is it a local isometry too? Take a nonsingular point
$u$ near $p$ so that $[up]$ is orthogonal to the singular component
containing $p$. The claim is that the length of $[up]$ is preserved under the projection ($p$ is projected into $x$ while $u$ is projected into some 
other point; we are only interested in one projection as the other projection
for $u$ and $p$ is the same: $p$). It is easy to see that $[up]$ is projected into a straight segment (i.e. a radial segment in a $2$-cone). Choose any $v \in [up]$ sufficiently close to $p$ and cover $[uv]$ with finitely many appropriate open neighborhoods --- the projection of $[uv]$ is (locally) a geodesic and hence a shortest path, since the projection is a radial segment. By continuity, the length of $[up]$ is preserved, too.

Now take two points $u$ and $u'$ near $p$ (still a codimension $2$ singular point) such that $[up]$ and $[u'p]$ are both orthogonal to the singular component
containing $p$ and project both onto the same cone as before (the other singular 
component). Draw segments along which we projected: $[uw]$ and $[u'w']$.
We already verified that the lengths stay the same: $|up| = |wx|$ and
$|u'p| = |w'x|$. Clearly, $[uw]$ and $[u'w']$ are parallel to the singular 
component along which we are projecting. If we start moving $u$ along
$[uw]$ and $u'$ along $[u'w']$, and also $p$ towards $x$, $|up|$ and $|u'p|$ stay the same and so the distance between $u$ and $u'$ locally stays the same, too!
Thus, by continuity (and compactness) the distance between $u$ and $u'$ is the
same as between $w$ and $w'$. This shows that our map is a local isometry
at codimension $2$ singularities as well.

Therefore, the map as defined on a pinched conical neighborhood of $x$ in $M$
(that is simply connected and a topological manifold) into a pinched
direct metric product of two appropriate cones is a local isometry and thus an
isometry.
\end{proof}

This completes the preliminary phase. The useful part of this analysis is summarized in the following lemma that will be used extensively in the final part of the argument.

\begin{lem}
\label{lem:local-product-decomposition-of-M}
Let $M$ be a PL-manifold of nonnegative curvature, homeomorphic to $S^2 \times S^2$.
Then at every point $p \in M$, $M$ can be locally represented in a unique way as
a product $C_1 \times C_2$ of two $2$-cones with conical angles $2 \pi \alpha_1 \leq 2 \pi$ and $2 \pi \alpha_2 \leq 2 \pi$ such that this decomposition is
aligned along our two 2-distributions. More precisely, for every
nonsingular point $(x, y)\in C_1 \times C_2$ near $p$ ($x \neq 0 \in C_1,
y \neq 0 \in C_2$) the two 2-distributions at $(x,y)$ are parallel to the
fibers $(C_1, \ast)$ and $(\ast, C_2)$, respectively. Lastly, there is
a uniform bound $\delta > 0$ such that every $p \in M$ has a neighborhood containing the ball $B_\delta(p)$ that again has a unique factoring with the factors aligned along the 2-distributions.
\end{lem}
\begin{proof}
The flat case is obvious. The codimension $2$ case is handled by 
Lemma~\ref{lem:codim-2-singularity-aligned}. 
There is no codimension $3$ case (Lemma~\ref{lem:no-codimension-3}).
The existence of factoring in the
codimension $4$ case is handled by Lemma~\ref{lem:classifying-codim-4-singularity}. To prove uniqueness, identify factors
as codimension $2$ singularities.

The ``lastly'' part is clear, since $M$ is a finite simplicial complex. Note that if  $\delta$ is sufficiently small, the factors will still be $C_1 \times C_2$, yet now $p$ need not be the vertex of either cone.
\end{proof}

\subsection{The decomposition}

Recall that the goal is to decompose $M$ into a direct metric product.
Lemma~\ref{lem:local-product-decomposition-of-M} gives a local decomposition.
The rest of the argument is very similar to the de Rham decomposition theorem
(it is crucial that $M = S^2 \times S^2$ is simply connected).
This is not surprizing --- the holonomies on the nonsingular part $M \backslash M_s$ respect this local factorization, as it is aligned along the two 2-distributions (parallel fields of oriented $2$-planes, orthogonal to each other) found in the first part of this paper. 

Fix one of the two $2$-distributions mentioned throughout the text and call it $\alpha$. We are going to learn to integrate this distribution not just on $M \backslash M_s$, but integrate it in some sense on $M$. 
Take any point $x \in M$, possibly a singular point. Construct a ``leaf''
 $L_{\alpha}(x) \subset M$ --- the smallest subset of $M$ containing $x$ and closed under a certain operation. Start from adding 
$x$ to $L_{\alpha}(x)$. 
Use the local decomposition at $x$ given by Lemma~\ref{lem:local-product-decomposition-of-M} and choose the
fiber parallel to $\alpha$. Take the points in $M$ near $x$ that belong to this fiber and add them to $L_{\alpha}(x)$, too. Continue this operation until every point in $L_{\alpha}(x)$ is there with a neighborhood of its appropriate fiber,
parallel to $\alpha$. 

The resulting set $L_{\alpha}(x)$ is a $2$-dimensional topological manifold
``immersed'' in $M$, called the leaf of $\alpha$ passing through $x$.

\begin{lem}
\label{lem:leaf-is-a-sphere}
Any such leaf $L_{\alpha}(x) \subset M$ is a compact simply-connected $2$-dimensional PL-manifold of nonnegative curvature (thus, a convex polyhedron). (This is ``without loss of generality'': we actually prove that this is true for any $L_{\alpha}(x)$ or for any $L_{\beta}(x)$. Here $\beta$ is the other 2-distribution that is orthogonal to $\alpha$.)
\end{lem}
\begin{proof}[Proof of lemma]
It is clear that $L_{\alpha}(x)$ is a $2$-dimensional PL-manifold of nonnegative
curvature from the way such leaves were defined. Using the lower bound $\delta$
in Lemma~\ref{lem:local-product-decomposition-of-M}, we see that $L_{\alpha}(x)$ has no boundary and is a complete metric space. The leaf is oriented as the
2-distribution $\alpha$ is oriented. To prove compactness we need the following:

\begin{clm}
There are no nonsingular leaves (in the sense of the intrinsic PL-metric). 
\end{clm}
\begin{proof}[Proof of claim]
Indeed, assume for a contradiction that $L_{\alpha}(x)$ is a flat
leaf in its intrinsic metric (while all its points may be singular in $M$).
Since the leaf has no boundary,
it may be a plane, a cylinder, or a flat torus. Every point $y \in L_{\alpha}(x) \subset M$ has a neighborhood from Lemma~\ref{lem:local-product-decomposition-of-M} that
contains the ball $B_\delta(y) \subset M$ (is not too small) and has a unique factoring,
where one of the factors is a neighborhood of $y$ in the leaf. Then the other
factor will be the same for all $y \in L_{\alpha}(x)$, and in a canonical way.
This is clear when the leaf is isometric to $\mathbb{R}^2$ and hence simply connected. For the cases of a
torus and a cylinder it becomes true if we view a torus (or a cylinder) as
the image of $\mathbb{R}^2$ ``immersed'' via a local isometry. 

This allows us to define a normal parallel field of directions on the leaf and move the leaf in this direction --- that is, any normal direction. 
Here it is crucial that $\delta$ is a uniform constant for all points in $M$, hence for all points in the leaf.
(Recall that we view a nonsingular leaf not just as a set, but as an ``immersion'' of a plane. So of course, during the movement a plane (the image) may become a torus, or vice versa.) What can be an obstacle for such an operation?

If the leaf is within distance $\delta$ from a codimension $4$ singularity or from a codimension $2$ singularity that is orthogonal (not parallel) to the leaf, the leaf itself must have a singularity (use 
Lemma~\ref{lem:local-product-decomposition-of-M}). Assume that it never happens.
If all codimension $2$ singularities are parallel to the leaf, without loss
of generality replace $\alpha$ with $\beta$, the orthogonal complement of $\alpha$. (So if we cannot prove the statement for any $L_{\alpha}(x)$, we will instead prove it for any $L_{\beta}(x)$.) 
Certainly, $M$ must have some
codimension $2$ singularities (it must have some singularities by the Gauss-Bonnet theorem, and then use Lemma~\ref{lem:flat-pinched-no-singularity} and Lemma~\ref{lem:no-codimension-3}). 
It is easy to argue that at any given moment all points of a leaf will be singularities of codimension $2$, or all points of a leaf will be nonsingular points from $M \backslash M_s$.
After moving the leaf around, it will 
span all of $M$ (contradicting the existence of singularities of codimension $2$ orthogonal to the leaf) or stop near such a singularity. Then the local factorization of $M$ near such a singularity will contradict the assumption that the leaf itself is nonsingular.
\end{proof}

So any leaf has singularities. However, any leaf can only have finitely many singularities. 
We can say that each singularity
``carries some angular defect'' that is the difference between the conical angle at this singularity
and $2 \pi$. Any such ``angular defect'' can only be a number from some fixed list of numbers between
$0$ and $2 \pi$, coming from the finite simplicial decomposition of $M$. In the case of a compact
leaf these angular defects add up to $4 \pi$. They cannot add up to more than that in the noncompact
case either. TO prove that, we can assume without loss of generality that the leaf is simply connected.
(If we conclude for a contradiction that the universal cover is compact, so is the leaf itself.)
A simply connected noncompact leaf is topologically a plane.
If the angular defects at different singularities 
add up to more than $2 \pi$, then the circumference 
of a sufficiently large circle around any point in this leaf (``plane'') 
decreases with some fixed rate as the radius increases, thus cannot
increase indefinitely. This implies finite diameter and hence compactness, leading 
to the desired contradiction. 

One can try to find a constant $D$ such that any point in $L_{\alpha}(x)$ is within
distance $D$ from some singularity in this leaf in the intrinsic metric.
If this is possible, choose a sequence of points $c_n \in L_{\alpha}(x)$ that
are further than $n$ from any singularity in this leaf. $M$ is compact,
so choose a converging subsequence $c_{n_k} \to c \in M$. Again using the local
decomposition of $M$ one can see that the leaf $L_{\alpha}(c)$ has no
singularities --- a contradiction.

So, every point in the leaf $L_{\alpha}(x)$ is not further than
$M$ from some singularity and there are finitely many such singularities --- say, $q$. The leaf has nonnegative curvature, so its area is at most $q \pi D^2 < \infty$. Finite
area clearly implies compactness. Compactness, nonnegative curvature and
orientability imply that the leaf is homeomorphic to $S^2$.
\end{proof}

We are going to focus on the set of all such leaves in $M$.

\begin{lem}
\label{lem:distances-between-leaves}
For any $x, y \in M$, 

$dist(L_{\alpha}(x), L_{\alpha}(y)) = dist(x, L_{\alpha}(y)) = dist(L_{\alpha}(x), y)$.
\end{lem}
\begin{proof}
Suffices to show that for all $x, \hat{x}, y \in M$ 
such that $L_{\alpha}(x) = L_{\alpha}(\hat{x})$, 
$dist(x, L_{\alpha}(y)) = dist(\hat{x}, L_{\alpha}(y))$. This can
be proved locally, for $x$ and $\hat{x}$ close to each other. 
The leaves are compact, so for a given $x$ we can find $z \in L_{\alpha}(y)$
closest to $x$: $dist(x, z) = dist(x, L_{\alpha}(y))$. Take any geodesic
$[xz]$. It arrives to $z$ parallel to one of the fibers of the local
product decomposition from Lemma~\ref{lem:local-product-decomposition-of-M}, for
otherwise it would not be a geodesic.
Hence, it goes along this fiber all the way from $x$ to $z$. 

Pick any $\hat{x}$ from the same leaf ($L_{\alpha}(x) = L_{\alpha}(\hat{x})$) that is close to $x$: $dist(x, \hat{x}) < \delta$ where $\delta$ is the constant
from Lemma~\ref{lem:local-product-decomposition-of-M}. Only choose 
$\hat{x}$ such that $dist(x, \hat{x}) = dist_{leaf}(x, \hat{x})$ --- they are
equally close in the intrinsic metric of the leaf. Then we can easily move
the geodesic $[xz]$ using the local product structure (chosen canonically
at all points) to obtain a segment $[\hat{x} \hat{z}]$ of the same length.
Therefore, $dist(\hat{x}, L_{\alpha}(y)) \le dist(x, L_{\alpha}(y))$, and
this implies what we need.
\end{proof}

Consequently, all leaves form a connected metric space $Leaves$ with the metric
\begin{equation*}
dist_{Leaves}(L_{\alpha}(x), L_{\alpha}(y)) =^{def}= dist_M (x, L_{\alpha}(y))
\end{equation*}
This metric is strictly intrinsic --- for two leaves $l_1$ and $l_2$ it is
easy to find $l_3$ in between: $dist_{Leaves}(l_1, l_2) / 2 =
dist_{Leaves}(l_1, l_3) = dist_{Leaves}(l_2, l_3)$.
Since $M$ is a disjoint union of different leaves,
this yields a natural mapping $M \mapsto Leaves$ where $x$ is sent to
$L_{\alpha}(x)$. This mapping is continious (because it is $1$-Lipshitz) and onto, so $Leaves$ is compact. $Leaves$ is also simply connected: any loop in 
$Leaves$ can be lifted to a path $\gamma:[0,1] \mapsto M$ with $\gamma(0)$
and $\gamma(1)$ in the same leaf: $L_{\alpha}(\gamma(0)) = L_{\alpha}(\gamma(1))$.
Connect $\gamma(0)$ with $\gamma(1)$  by a path in this leaf and contract the 
resulting loop in $M$.

This allows us to sharpen the statement of Lemma~\ref{lem:local-product-decomposition-of-M}. 

\begin{lem}
\label{lem:improved-local-product-decomposition-of-M}
There is $\epsilon > 0$ (smaller than $\delta$ from Lemma~\ref{lem:local-product-decomposition-of-M}) such that
for every point $p \in M$, a neighborhood
of $p$ in $M$ can be factored (again along the 2-distributions) into
a product of the $\epsilon$-neighborhood of $p$ in $L_{\alpha}(p)$ with
the $\epsilon$-neighborhood of $L_{\alpha}(p)$ in $Leaves$.
\end{lem}
\begin{proof}
Since all leaves are compact, they have finite area and a sufficiently small neighborhood of $p$ will intersect $L_{\alpha}(p)$ only along the fiber of
the decomposition parallel to $\alpha$ (and not along several parallel fibers).
Choose such an $\epsilon$ and use $\epsilon / 3$ to make sure the distances between different leaves measured within this neighborhood are indeed true distances.

It is easy to choose a uniform $\epsilon$ for all points, as the maximal $\epsilon$ that works for a given $p$ is a $3$-Lipschitz function of $p$, and 
$M$ is compact. All details follow easily from Lemma~\ref{lem:distances-between-leaves}.
\end{proof}

It remains to argue that this gives us the desired decomposition.

\begin{thm}
$M$ is a direct metric product of any leaf $L_{\alpha}(x)$ with the space of all leaves, $Leaves$.
\end{thm}
\begin{proof}
Let $\epsilon$ be as in Lemma~\ref{lem:improved-local-product-decomposition-of-M}.
Pick any leaf $l \in Leaves$ and let $U$ be the $\epsilon / 2$-neighborhood
of $l$ in $Leaves$. Let $Z = f^{-1}(U) \subset M$ be the set of all points
in $M$ that are closer to the leaf $l$ than $\epsilon / 2$. Here $f$ is
the projection $M \mapsto Leaves$ sending $x$ to $L_{\alpha}(x)$.
It is clear from Lemma~\ref{lem:distances-between-leaves} and 
Lemma~\ref{lem:improved-local-product-decomposition-of-M} that
for every $x \in Z$ there is exactly one $y \in l$ closest to $x$
($dist(x, l) = dist(x, y)$). Hence, $Z$ is homotopy equivalent to $l$ and
as such is simply connected. 

Lemma~\ref{lem:improved-local-product-decomposition-of-M} gives a local
isometry of $Z$ with $l \times U$ (this isometry is well-defined as
$l$ is simply connected), and this is also a global isometry (as
both spaces are simply connected --- $U$ is an $\epsilon / 2$ - neighborhood
of a point in a $2$-cone). 

This implies that all leaves are isometric ($l$ is isometric to all leaves
in $U$, and $Leaves$ is connected). Any curve in $Leaves$ defines an
isometry between the two leaves it connects. This isometry is
trivial for a closed curve that is shorter than $\epsilon / 2$ and
$Leaves$ is simply connected. Hence, all leaves are isometric to
each other in a canonical way (fix a leaf $l_0$ and connect it
to ever other leaf via any curve). So $M$ is locally isometric
to $l_0 \times Leaves$ and both sides are simply connected, hence it
is indeed a true isometry.
\end{proof}

We have established that $M \simeq L \times Leaves$, where $L$ is a convex
polyhedron in $\mathbb{R}^3$ (Lemma~\ref{lem:leaf-is-a-sphere}). $Leaves$ is a PL $2$-dimensional manifold of nonnegative curvature (from local product structure --- Lemma~\ref{lem:improved-local-product-decomposition-of-M}). We also know that $Leaves$ is connected, simply connected and compact
(see remarks after the space $Leaves$ was defined). So it is also a convex
polyhedron in $\mathbb{R}^3$, and we are done.

\bibliographystyle{alpha}
\bibliography{pl-bio}

\end{document}